\documentclass[12pt]{article}

\usepackage{amssymb}
\usepackage{amsfonts}

\usepackage{dsfont}

\usepackage{amsmath, amsthm}

\newtheorem{thm}{Theorem}
\newtheorem{lem}{Lemma}
\theoremstyle{definition}
\newtheorem{defin}{Definition}
\newtheorem{rem}{Remark}

\DeclareMathOperator{\dist}{dist} 
\DeclareMathOperator{\LipPerSh}{LipPerSh}\DeclareMathOperator{\LipSh}{LipSh}
\DeclareMathOperator{\diag}{diag}
\DeclareMathOperator{\Cl}{Cl}\DeclareMathOperator{\Per}{Per}
\DeclareMathOperator{\Id}{Id}

\newcommand{\ep}{\varepsilon}
\newcommand{\al}{\alpha}
\newcommand{\lam}{\lambda}
\newcommand{\ZZ}{\mathds{Z}}
\newcommand{\NN}{\mathds{N}}
\newcommand{\RR}{\mathds{R}}
\newcommand{\sref}[1]{(\ref{#1})}
\newcommand{\LL}{\mbox{$\cal{L}$}}
\newcommand{\be}{\beta}
\newcommand{\D}{D}
\newcommand{\CR}{{\cal C}{\cal R}}

\title{Lipschitz Shadowing for Flows}
\author{Kenneth J. Palmer\thanks{Corresponding author, partially supported by NSC (Taiwan)  97-2115-M-002 -011 -MY2}  \\
Department of Mathematics,\\
National Taiwan University, Taipei 106 - Taiwan
\and Sergei Yu. Pilyugin\thanks{partially supported by NSC (Taiwan)  97-2115-M-002 -011 -MY2}  \\
Faculty of Mathematics and Mechanics,\\
St. Petersburg State University,\\
University av. 28, 198504, St. Petersburg, Russia\\
\and Sergey B. Tikhomirov\thanks{partially supported by NSC (Taiwan)  97-2115-M-002 -011 -MY2 and CNPq (Brazil)}  \\
Departamento de Matematica,\\ Pontificia Universidade Catolica do Rio de Janeiro\\
Rua Marques de Sao Vicente, 225, Rio de Janeiro, Brazil.}
\date{\today}

\begin{document}

\maketitle

\begin{abstract}
Let $\phi$ be the flow generated by a smooth vector field $X$ on a
smooth closed manifold. We show that the Lipschitz shadowing
property of $\phi$ is equivalent to the structural stability of $X$
and that the Lipschitz periodic shadowing property of $\phi$ is
equivalent to the $\Omega$-stability of $X$.
\end{abstract}

{\bf keyword:}  vector fields; Lipschitz shadowing; periodic shadowing;  structural stability\\
{\bf AMS(MOS) subject classifications.}   37C50, 37D20\\

%\end{frontmatter}

\section{Introduction}

The theory of shadowing of approximate trajectories
(pseudotrajectories) of dynamical systems is now a well developed
part of the global theory of dynamical systems (see, for example,
the monographs \cite{Palmer1}, \cite{Pilyugin}). This theory is
closely related to the classical theory of structural stability (the
basic definitions of structural stability and $\Omega$-stability for
flows can be found, for example, in the monograph \cite{PilSSBook}).
It is well known that a diffeomorphism has the shadowing property in
a neighborhood of a hyperbolic set \cite{Anosov}, \cite{Bowen} and a
structurally stable diffeomorphism has the shadowing property on the
whole manifold \cite{Robinson}, \cite{Morimoto}, \cite{Sawada}.
Analyzing the proofs of the first shadowing results by Anosov
\cite{Anosov} and Bowen \cite{Bowen}, it is easy to see that, in a
neighborhood of a hyperbolic set, the shadowing property is
Lipschitz (and the same holds in the case of a structurally stable
diffeomorphism, see \cite{Pilyugin}). At the same time, it is easy
to give an example of a diffeomorphism that is not structurally
stable but has the shadowing property (see \cite{Pilyugin2}, for
example). Thus, structural stability is not equivalent to shadowing.
However it was shown in \cite{Pilyugin1} that structural stability
of a diffeomorphism is equivalent to Lipschitz shadowing.

Turning to flows, it is well known that a flow has the shadowing
property in a neighborhood of a hyperbolic set \cite{Palmer1},
\cite{Pilyugin} and a structurally stable flow has the shadowing
property on the whole manifold \cite{Pilyugin}, \cite{Pilyugin4}. In
fact, in a neighborhood of a hyperbolic set, the shadowing property
is Lipschitz and the same holds in the case of a structurally stable
flow, see \cite{Pilyugin}. At the same time, it is easy to give an
example of a flow that is not structurally stable but has the
shadowing property (to construct such an example, one can use almost
the same idea as in \cite{Pilyugin2}). Thus, as with
diffeomorphisms, structural stability is not equivalent to
shadowing. However it is our purpose in this article to show that
structural stability of a flow is equivalent to Lipschitz shadowing.
Let us note that the proof for the flow case is a nontrivial
modification of the proof for the diffeomorphism case.

One of the previously used approaches to compare shadowing property
and structural stability is passing to $C^1-$interiors. Sakai
\cite{Sak} showed that the $C^1-$interior of the set of
diffeomorphisms with the shadowing property coincides with the set
of structurally stable diffeomorphisms. See also \cite{PilRodSak}
for the generalization of this result to other types of shadowing
properties. For vector fields the situation is different. There is
an example of a vector field with the robust shadowing property
which is not structurally stable \cite{PilTikh}. See also
\cite{LeeSak,PilTikh2008,Tikh} for some positive results in this
direction.

In this paper, we also study vector fields having the Lipschitz
periodic shadowing property. Diffeomorphisms having the Lipschitz
periodic shadowing property were studied in \cite{OPT}, where it was
shown that this property is equivalent to $\Omega$-stability. We
prove a similar statement for vector fields.

\section{Preliminaries}

Let $M$ be a smooth closed manifold with Riemannian metric
dist$(\cdot,\cdot)$ and let $X$ be a vector field on $M$ of class
$C^1$. Let $\phi(t,x)$ be the flow on $M$ generated by $X$.
\medskip

\begin{defin}   A (not necessarily continuous) function
$y:I\to M$ (where $I$ is an interval in $\RR$) is called a
$d$-pseudotrajectory if
$$
\dist(y(\tau+t),\phi(t,y(\tau))\leq d,\quad 0\leq t \leq 1, \quad
\tau, \tau+t\in I.
$$
Mostly we work with pseudotrajectories defined on $I=\RR$.
\end{defin}

\begin{defin}  We say that the vector field $X$ has the {\it Lipschitz shadowing property}
($X \in \LipSh$) if there exist $d_0$ and ${\cal L}>0$ such that if
$y:\RR\mapsto M$ is a $d$-{\it pseudotrajectory} for $d\le d_0$,
then $y(t)$ is ${\cal L}d$-{\it shadowed by a trajectory}, that is,
there exists a trajectory $x(t)$ of $X$ and an increasing
homeomorphism (\textit{reparametrization}) $\alpha(t)$ of the real
line satisfying
\begin{equation}
\label{eqdef2.1} \al(0) = 0, \quad
\left|\frac{\alpha(t_2)-\alpha(t_1)}{t_2-t_1}-1\right|\le {\cal L}d
\end{equation}
for $t_2\neq t_1$ and
\begin{equation}\label{eqdef2.2}
\dist(y(t),x(\alpha(t))\le {\cal L}d
\end{equation}
for all $t$.
\end{defin}

\begin{defin} We say that the vector field $X$ has the
\textit{Lipschitz periodic shadowing property} ($X \in \LipPerSh$)
if there exist $d_0$ and ${\cal L}  > 0$ such that if $y:\RR\mapsto
M$ is a periodic $d$-{\it pseudotrajectory} for $d\le d_0$, then
$y(t)$ is ${\cal L}d-${\it shadowed by a periodic trajectory}, that
is, there exists a trajectory $x(t)$ of $X$ and an increasing
homeomorphism $\alpha(t)$ of the real line satisfying inequalities
\sref{eqdef2.1} and \sref{eqdef2.2} and such that
\begin{equation}
\notag x(t+\omega)=x(t)
\end{equation}
for some $\omega > 0$.

The last equality implies that $x(t)$ is either a closed trajectory
or a rest point of the flow $\phi$.
\end{defin}

The main results of the paper are the following theorems.

\begin{thm}
\label{thmLipSh} A vector field $X$ satisfies the Lipschitz
shadowing property if and only if $X$ is structurally stable.
\end{thm}

\begin{thm}
\label{thmLipPerSh} A vector field $X$ satisfies the Lipschitz
periodic shadowing property if and only if $X$ is $\Omega$-stable.
\end{thm}

It is known that expansive diffeomorphisms having the Lipschitz
shadowing property are Anosov (see \cite{Pilyugin1}).

We show, as a consequence of Theorem \ref{thmLipSh}, that expansive
vector fields having the Lipschitz shadowing property are Anosov.
Let us recall the definition of expansivity for vector fields.

\begin{defin} We say that a vector field $X$
and the corresponding flow $\phi(t,x)$ are \textit{expansive} if
there exist constants $a, \delta > 0$ such that if
$$
\mbox{dist}(\phi(t, x), \phi(\al(t), y)) < a, \quad t \in\RR,
$$
for points $x, y \in M$ and an increasing homeomorphism $\al$ of the
real line, then $y = \phi(\tau, x)$ for some $|\tau| < \delta$.
\end{defin}

\begin{thm}
An expansive vector field $X$ having the Lipschitz shadowing
property is Anosov.
\end{thm}

\begin{proof} By Theorem \ref{thmLipSh}, a vector field $X$ having the Lipschitz
shadowing property is structurally stable. Hence, there exists a
neighborhood ${\cal N}$ of $X$ in the $C^1$-topology such that any
vector field in ${\cal N}$ is expansive (this property of $X$ is
sometimes called robust expansivity).

By Theorem B of \cite{MorSakSun}, robustly expansive vector fields
having the shadowing property are Anosov.
\end{proof}

In Sec.~\ref{secLipSh} we prove Theorem~\ref{thmLipSh} and in
Sec.~\ref{secLipPerSh} we prove Theorem~\ref{thmLipPerSh}. Both
proofs are long so that each section is divided into several
subsections.

\section{The Lipschitz shadowing property}\label{secLipSh}

As was mentioned above in \cite{Pilyugin4} it was proved that
structurally stable vector fields have the Lipschitz shadowing
property. Our goal here is to show that vector fields satisfying
Lipschitz shadowing are structurally stable. It is well known (see
\cite{Robinson1}) that for this purpose it is enough to show that
such a vector field satisfies Axiom A$'$ and the strong
transversality condition.

First we show that Lipschitz shadowing implies discrete Lipschitz
shadowing. Define a diffeomorphism $f$ on $M$ by setting
$f(x)=\phi(1,x)$.
\medskip

\begin{defin}\label{defDLSh} The vector field $X$ has the {\it discrete Lipschitz shadowing
property} if there exist $d_0$, $L>0$ such that if $y_k\in M$ is a
sequence with
$$ \dist(y_{k+1},f(y_k)) \le d, \quad k\in \ZZ$$
for $d\le d_0$, then there exist sequences $x_k\in M$ and $t_k\in
\RR$ satisfying
$$  |t_k-1|\le Ld,\quad \dist(x_k, y_k) \le Ld, \quad x_{k+1}=\phi(t_k,x_k)$$
for all $k$.
\end{defin}
\medskip

\begin{lem}\label{lemPalm1}  Lipschitz shadowing implies discrete Lipschitz
shadowing.
\end{lem}

\begin{proof}  Let $y_k$ be a sequence with
$$ {\rm dist}(y_{k+1},f(y_k))={\rm dist}(y_{k+1},\phi(1,y_k))\le d, \quad k\in \ZZ.$$
Then we define
$$ y(t)=\phi(t-k,y_k)\quad k\le t<k+1,\quad k\in \ZZ.$$
Assume that $k\le\tau<k+1$. If $0\le t\le 1$ and $\tau+t<k+1$, then
$$ {\rm dist}(y(\tau+t),\phi(t,y(\tau))
={\rm dist}(\phi(\tau+t-k,y_k),\phi(t,\phi(\tau-k,y_k)))=0$$
and if $k+1\le \tau+t$, then
$$ \begin{array}{rl}
&{\rm dist}(y(\tau+t),\phi(t,y(\tau)))\\
&={\rm dist}(\phi(\tau+t-k-1,y_{k+1}),\phi(t+\tau-k,y_k))\\
&={\rm dist}(\phi(\tau+t-k-1,y_{k+1}),\phi(\tau+t-k-1,\phi(1,y_k)))\\
&\le \nu d,\end{array}$$
where $\nu$ is a constant such that
\begin{equation}\label{new5}{\rm dist}(\phi(t,x),\phi(t,y))\le \nu\,{\rm dist}(x,y)\quad{\rm for}\quad
x,y \in M,\; 0\le t\le 1.\end{equation}

Then if $d\le d_0/\nu$, there exists a trajectory
$x(t)$ of $X$ and a function $\alpha(t)$ satisfying
$$ \left|\frac{\alpha(t_2)-\alpha(t_1)}{t_2-t_1}-1\right|\le {\cal L}\nu d $$
for $t_2\neq t_1$ and
$$ {\rm dist}(y(t),x(\alpha(t)))\le {\cal L}\nu d$$
for all $t$. Then if we define
$$ x_k=x(\alpha(k)),\quad t_k=\alpha(k+1)-\alpha(k),$$
we see that
$$ x_{k+1}=x(\alpha(k+1))=\phi(\alpha(k+1)-\alpha(k),x(\alpha(k)))
=\phi(t_k,x_k),$$
$$ \dist(x_k, y_k)=\dist(x(\alpha(k)), y(k)) \le {\cal L}\nu d$$
and
$$ |t_k-1|=\left|\frac{\alpha(k+1)-\alpha(k)}{k+1-k}-1\right|
\le {\cal L}\nu d.$$

Taking $L = {\cal L}\nu$ and $d_0$ in Definition \ref{defDLSh} as
$d_0/\nu$,  we complete the proof of the lemma.
\end{proof}

Our main tool in the proof is the following lemma which
relates the shadowing problem to the problem of existence
of bounded solutions of certain difference equations.
To ``linearize'' our problem, we apply the standard technique of exponential mappings.

Denote by $T_xM$ the tangent space to $M$ at a point $x$; let $|v|$ be the norm of $v$ corresponding to the metric dist$(\cdot,\cdot)$.

Let $\exp:TM\mapsto M$ be the standard exponential mapping on the
tangent bundle of $M$ and let $\exp_x$ be the corresponding mapping
$T_xM\mapsto  M$.

Denote by $B(r,x)$ the ball in $M$ of radius $r$ centered at a point
$x$ and by $B_T(r,x)$ the ball in $T_xM$ of radius $r$ centered at
the origin.

There exists $r>0$ such that, for any $x\in M$, $\exp_x$ is a
diffeomorphism of $B_T(r,x)$ onto its image, and $\exp_x^{-1}$ is a
diffeomorphism of $B(r,x)$ onto its image. In addition, we may
assume that $r$ has the following property:

If $v$, $w\in B_T(r,x)$, then
\begin{equation}\label{new3}
{\rm dist}(\exp_x(v), \exp_x(w)) \le 2 |v-w|;
\end{equation}
if $y$, $z\in B(r,x)$, then
\begin{equation}\label{new4}
|\exp_x^{-1}(y)-\exp_x^{-1}(z)| \le 2 \dist(y, z).
\end{equation}

Let $x(t)$ be a trajectory of $X$; set $p_k = x(k)$ for $k \in \ZZ$.
Denote $A_k=Df(p_k)$ and ${\cal M}_k=T_{p_k}M$. Clearly, $A_k$ is a
linear isomorphism between ${\cal M}_k$ and ${\cal M}_{k+1}$.

In the sequel whenever we construct $d$-pseudotrajectories of the
diffeomorphism $f$, we always take $d$ so small that the points of
the pseudotrajectories under consideration, the points of the
associated shadowing trajectories, their lifts to tangent spaces,
etc. belong to the corresponding balls $B(r,p_k)$ and $B_T(r,p_k)$.

We consider the mappings
\begin{equation}\label{new1}
F_k=\exp_{p_{k+1}}^{-1}\circ f\circ \exp_{p_{k}}: B_T(\rho,p_k)\to
{\cal M}_{k+1}
\end{equation}
with $\rho\in(0,r)$ small enough, so that
$$
f\circ \exp_{p_{k}}(B_T(\rho,p_k))\subset B(r,p_{k+1}).
$$
It follows from standard properties of the exponential mapping that
$D\exp_x(0)= \Id$; hence,
$$ DF_k(0)=A_k.$$
Since $M$ is compact, for any $\mu>0$ we can find
$\delta=\delta(\mu) >0$ such that if $|v|\le\delta$, then
\begin{equation}\label{new2}|F_k(v)-A_kv|\le\mu|v|.
\end{equation}

\begin{lem}\label{lemPalm2} {\it Assume that $X$ has the discrete Lipschitz shadowing
property with constant $L$. Let $x(t)$ be an arbitrary trajectory of
$X$, let $p_k = x(k)$, $A_k=Df(p_k)$ and let $b_k\in{\cal M}_k$ be a
bounded sequence (denote $b=\|b\|_{\infty}$). Then there exists a
sequence $s_k$ of scalars with $|s_k|\le b'=L(2b+1)$ such that the
difference equation
$$ v_{k+1}=A_kv_k+X(p_{k+1})s_k+ b_{k+1} $$
has a solution $v_k$ such that}
$$ \|v\|_{\infty}\le 2b'.$$
\end{lem}

\begin{proof} Fix a natural number $N$ and define $\Delta_k\in{\cal M}_k$ as
the solution of
$$ v_{k+1}=A_kv_k+b_{k+1},\quad k=-N,\ldots,N-1 $$
with $\Delta_{-N}=0$.  Then
\begin{equation}\label{4} |\Delta_k|\le C,\end{equation}
where $C$ depends on $N$, $b$ and an upper bound on $|A_k|$.

Fix a small number $d>0$ and fix $\mu$ in (\ref{new2}) so that $\mu
< 1/(2C)$. Then consider the sequence of points $y_k\in M$, $k\in
\ZZ$, defined as follows: $y_k = p_k$ for $k\le -N$, $y_k =
\exp_{p_k}(d\Delta_k)$ for $-N + 1\le k \le N$, and $y_{N +k} =
f^k(y_N)$ for $k > 0$.

By definition, $y_{k+1} = f(y_k)$ for $k\le -N$ and $k\ge N$. If
$-N-1 \le k \le N - 1$, then
$$y_{k+1} = \exp_{p_{k+1}}(d\Delta_{k+1})= \exp_{p_{k+1}}(dA_k\Delta_k + db_{k+1}),$$
and it follows from estimate (\ref{new3}) that if $d$ is small enough, then
\begin{equation}\label{6} {\rm dist}\left(y_{k+1}, \exp_{p_{k+1}}(dA_k\Delta_k )\right) \le 2d|b_{k+1}| \le 2db. \end{equation}
On the other hand,
$$ f(y_k) = \exp_{p_{k+1}}(F_k(d\Delta_k ))$$
(see the definition (\ref{new1}) of the mapping $F_k$), and we deduce from (\ref{new3}),
(\ref{new2}) and (\ref{4}) that if $Cd\le\delta(\mu)$
\begin{equation}\label{7} \begin{array}{rl}
{\rm dist}\left(f(y_k), \exp_{p_{k+1}}(dA_k\Delta_k)\right)
&\le 2|F_k(d\Delta_k)-dA_k\Delta_k|\\
&\le 2\mu|d\Delta_k|\\
&\le 2C\mu d\\
&< d.
\end{array}\end{equation}
Estimates (\ref{6}) and (\ref{7}) imply that
$$
{\rm dist}(y_{k+1}, f(y_k )) < d(2b+1),\quad k \in \ZZ,
$$
if $d$ is small enough (let us emphasize here that the required
smallness of $d$ depends on $b$, $N$ and estimates on $A_k$). By
hypothesis, there exist sequences $x_k$ and $t_k$ such that
$$ |t_k-1| \le b'd,\; {\rm dist}(x_k, y_k) \le b'd,\;
x_{k+1} = \phi(t_k ,x_k ),\quad k \in \ZZ.$$ If we write
$$ x_k = \exp_{p_k}(dc_k),\; t_k = 1 + ds_k,$$
then it follows from estimate (\ref{new4}) that
$$|dc_k-d\Delta_k | \le 2\,{\rm dist}(x_k, y_k) \le 2b'd.$$
Thus,
\begin{equation}\label{9}|c_k-\Delta_k| \le 2b',\; k \in \ZZ.
\end{equation}
Clearly,
\begin{equation}\label{10} |s_k| \le b',\; k \in \ZZ.
\end{equation}

%\textbf{ONLY IN NEIGHBORHOODS} Define a mapping $G_k: \RR \times
%{\cal M}_k\mapsto {\cal M}_{k+1}$ by
%$$ G_k(t,v) = \exp^{-1}_{p_{k+1}}(\phi(1+t,\exp_{p_k}(v))).$$

We may assume that the value $\rho$ fixed above is small enough, so
that the mappings
$$
G_k:(-\rho,\rho)\times B_T(\rho,p_k)\to {\cal M}_{k+1}
$$
given by
$$
G_k(t,v) = \exp^{-1}_{p_{k+1}}(\phi(1+t,\exp_{p_k}(v))).
$$
are defined. Then $G_k(0,0)=0$,
\begin{equation}\label{11} D_tG_k(t,v)|_{t=0,v=0} = X(p_{k+1}),\quad
D_vG_k(t,v)|_{t=0,v=0} = A_k.\end{equation}
We can write the equality
$$ x_{k+1} = \phi(1 + ds_k, x_k)$$
in the form
$$ \exp_{p_{k+1}}(dc_{k+1}) = \phi(1 + ds_k, \exp_{p_k}(dc_k )),$$
which is equivalent to
\begin{equation}\label{12} dc_{k+1} = G_k(ds_k,dc_k).\end{equation}
Now let $d = d_m$, where $d_m \to 0$. Note that the corresponding
$c_k = c_k^{(m)}$, $t_k = t_k^{(m)}$, and $s_k = s_k^{(m)}$ depend
on $m$.

Since $|c_k^{(m)}| \le 2b' +C$ and $|s_k^{(m)}| \le b'$ for all $m\ge  1$ and $-N \le k \le N-1$, by taking a subsequence if necessary, we can assume that $c_k^{(m)}\to \tilde c_k$, $t_k^{(m)}\to \tilde t_k$,
and $s_k^{(m)}\to  \tilde s_k$ for $-N \le k \le N - 1$ as $m\to\infty$.

Applying relations (\ref{12}) and (\ref{11}), we can write
$$ d_mc_{k+1}^{(m)} = G_k(d_ms_k^{(m)},d_mc_k)
=  A_kd_mc_k^{(m)}+ X(p_{k+1})d_ms_k^{(m)}+ o(d_m).$$
Dividing by $d_m$, we get the relations
$$ c_{k+1}^{(m)} = A_kc_k^{(m)}+X(p_{k+1})s_k^{(m)} +  o(1),\;
-N \le k \le N - 1.$$
Letting $m\to\infty$, we arrive at
$$\tilde c_{k+1} =A_k\tilde c_k + X(p_{k+1})\tilde s_k,\; -N \le k \le N - 1,$$
where
$$ |\Delta_k -\tilde c_k| \le 2b',\quad |\tilde s_k|\le b',\quad-N \le k \le N - 1$$
due to (\ref{9}) and (\ref{10}).

Denote the obtained $\tilde s_k$ by $s_k^{(N)}$. Then $v_k^{(N)} = \Delta_k - \tilde c_k$ is a solution of the system
$$ v^{(N)}_{k+1} = A_kv^{(N)}_k + X(p_{k+1})s_k^{(N)}+ b_{k+1},\;
-N \le k \le N - 1,$$
such that $|v^{(N)}_k| \le 2b'$.

There exist subsequences $s_k^{(j_N)}\to s_k'$ and $v_k^{(j_N)}\to v_k'$ as $N\to\infty$ (we do
not assume uniform convergence) such that $|s_k'| \le b'$, $|v_k'|\le 2b'$, and
$$ v_{k+1}' =  A_kv_k' + X(p_{k+1})s_k' + b_{k+1},\; k\in \ZZ.$$
Thus, the lemma is proved.
\end{proof}

Further, we have to refer to two known statements. It is convenient to state them as lemmas.
First we make a definition.

\begin{defin}\label{defPalm4} Consider a sequence of linear isomorphisms
$$ C = \{C_k: \RR^n \mapsto \RR^n, k\in \ZZ\}$$
such that $\sup_k(\|C_k\|, \|C_k^{-1}\|) <\infty$. The associated
{\it transition operator} is defined for indices $k$, $l\in \ZZ$ by
$$ \Phi(k,l)=\begin{cases}C_{k-1}\circ\cdots\circ C_l, & l < k,\cr
                    \Id, & l=k,\cr
                    C_{k}^{-1}\circ\cdots\circ C^{-1}_{l-1}, & l >
                    k.
                    \end{cases}
                    $$
The sequence $C$ is called {\it hyperbolic on} $\ZZ_+$ (has an {\it
exponential dichotomy} on $\ZZ_+$) if there exist constants $K > 0$,
$\lambda \in (0,1)$, and families of linear subspaces $S_k$, $U_k$
of $\RR^n$ for $k\in \ZZ_+$ such that
\begin{description}
\item{(1)} $S_k\oplus U_k = \RR^n$, $k \in \ZZ_+$;
\item{(2)} $C_k(S_k) = S_{k+1}$ and $C_k(U_k) = U_{k+1}$ for $k\in \ZZ_+$;
\item{(3)} $|\Phi(k,l)v|\le K\lambda^{k-l}|v|$ for $v \in S_l$, $k \ge l \ge 0$;
\item{(4)} $|\Phi(k, l)v|\le K\lambda^{l-k}|v|$ for $v \in U_l$, $0\le k\le l$.
\end{description}
\end{defin}
The following result was shown by Maizel$'$ \cite{Maizel} (see also
Coppel  \cite{Coppel}).
\medskip

\begin{lem}\label{lemPalm3} If the system
$$ v_{k+1} = C_kv_k + b_{k+1},\quad k \ge 0,$$
has a bounded solution $v_k$ for any bounded sequence $b_k$, then
the sequence $C$ is hyperbolic on $\ZZ_+$ (and a similar statement
holds for $\ZZ_-$).
\end{lem}

The second of the results which we need was proved by Pliss \cite{Pliss}. An analogous statement was proved later by
Palmer \cite{Palmer2, Palmer3}; he also described the Fredholm properties of the corresponding operator
$$ \{v_k\in \RR^m: k\in \ZZ\}\mapsto \{v_k-A_{k-1}v_{k-1}\}.$$
\medskip

\begin{lem}\label{lemPalm4}  Set
$$ B^+(C) = \{v \in \RR^n: |\Phi(k, 0)v| \to 0,\; k \to +\infty\}$$
and
$$ B^-(C) = \{v \in \RR^n: |\Phi(k, 0)v| \to 0,\; k \to -\infty\}.$$
Then the following two statements are equivalent:
\begin{description}
\item{{\rm (a)}} for any bounded sequence $\{b_k \in \RR^n,\; k \in \ZZ\}$ there exists a bounded sequence $\{v_k \in \RR^n,\; k \in \ZZ\}$ such that
$$ v_{k+1} =C_kv_k + b_{k+1},\quad k \in \ZZ;$$
\item{{\rm (b)}} the sequence $C$ is hyperbolic on each of the rays $\ZZ_+$ and $\ZZ_-$, and the subspaces $B^+(C)$ and $B^-(C)$ are transverse.
\end{description}
\end{lem}

\begin{rem}\label{remPalm1} Both Lemmas \ref{lemPalm3} and \ref{lemPalm4} were proved for linear systems of
differential equations, but they hold as well (in the form stated
above) for sequences of linear isomorphisms of Euclidean spaces and
for sequences of linear isomorphisms of arbitrary linear spaces of
the same dimension (we apply them to the isomorphisms $A_k$ of the
spaces ${\cal M}_k$ in Section \ref{secPalm4} and to the
isomorphisms $B_k$ of the spaces $V_k$ in Sections \ref{secPalm5}
and \ref{secPalm6}). For further discussion of this point, see
\cite{Pilyugin3}.
\end{rem}

In the following three sections we assume that $X$ has the Lipschitz
shadowing property (and, consequently, the discrete Lipschitz
shadowing property).

\subsection{Hyperbolicity of the rest points}
\label{secPalm3}

Let $x_0$ be a rest point.  We apply Lemma \ref{lemPalm2} with
$p_k=x_0$. Noting that $X(p_k)=0$, we conclude that the difference
equation
$$ v_{k+1}=Df(x_0)v_k+b_{k+1}$$
has a bounded solution $v_k$ for all bounded sequences $b_k\in{\cal
M}_{x_0}$. Then it follows from Lemma \ref{lemPalm4} that
$$ v_{k+1}=Df(x_0)v_k $$
is hyperbolic on both $\ZZ_+$ and $\ZZ_-$.  In particular, this
implies that any solution bounded on $\ZZ_+$ tends to $0$ as
$k\to\infty$. However if $Df(x_0)$ had an eigenvalue on the unit
circle, the equation would have a nonzero solution with constant
norm. Hence the eigenvalues of $Df(x_0)$ lie off the unit circle. So
$x_0$ is hyperbolic.

\subsection{The rest points are isolated in the chain recurrent
set}\label{secPalm4}

\begin{lem}\label{lemPalm5} If a rest point $x_0$ is not isolated in the chain
recurrent set ${\cal CR}$, then there is a homoclinic orbit $x(t)$
associated with it.
\end{lem}

\begin{proof} We choose $d>0$ so small that ${\rm dist}(\phi(t,y),x_0)\le
{\cal L}d$ for $|t|$ large implies that $\phi(t,y)\to x_0$ as
$|t|\to\infty$.

Assume that there exists a point $y\in {\cal CR}$ such that $y\neq
x_0$ is arbitrarily close to $x_0$. Since $y$ is chain recurrent,
given any $\ep_0$ and $\theta>0$ we can find points $y_1,\dots,y_N$
and numbers $T_0,\dots,T_N>\theta$ such that
$$
\dist(\phi(T_0,y),y_1)<\ep_0,
$$
$$
\dist(\phi(T_i,y_i),y_{i+1})<\ep_0,\quad i=1,\dots,N,
$$
$$
\dist(\phi(T_N,y_N),y)<\ep_0.
$$
Set $T=T_0+\dots+T_N$ and define $g^*$ on $[0,T]$ by
$$
g^*(t) =
\begin{cases}
\phi(t,y),&\quad 0\leq t< T_0,\\
\phi(t,y_i),&\quad T_0+\dots+T_{i-1}\leq t <T_0+\dots+T_{i},\\
y,&\quad t=T.
\end{cases}
$$
Clearly, for any $\ep>0$ we can find $\ep_0$ depending only on $\ep$
and $\nu$ (see \sref{new5}) such that $g^*(t)$ is an
$\ep$-pseudotrajectory on $[0,T]$.

Then we define
$$
g(t)= \begin{cases}x_0, & t\le 0, \cr
  g^*(t),  & 0<t\le T, \cr
  x_0, & t>T.
\end{cases}
$$
We want to choose $y$ and $\varepsilon$ in such a way that $g(t)$ is
a $d$-pseudotrajectory. We need to show that for all $\tau$ and
$0\le t\le 1$
\begin{equation}\label{ST1}{\rm dist}(\phi(t,g(\tau)),g(t+\tau))\le d.\end{equation}

Clearly this holds for (i) $\tau\le -1$, (ii) $\tau\ge T$, (iii) $\tau$, $\tau+t\in [-1,0]$, and
(iv) $\tau$, $\tau+t\in [0,T]$.

If $-1\le \tau\le 0$, $\tau+t>0$, then with $\nu$ as in \sref{new5}
$$
\begin{array}{rl}
 {\rm dist}(\phi(t,g(\tau)),g(\tau+t))
&={\rm dist}(x_0,g^*(\tau+t))\\
&\le {\rm dist}(x_0,\phi(\tau+t,y))+{\rm dist}(\phi(\tau+t,y),g^*(\tau+t))\\
&\le \nu{\rm dist}(x_0,y)+\varepsilon\\
&\le d,\end{array}
$$
if $\dist(y,x_0)$ and $\varepsilon$ are
sufficiently small. Note that, for the fixed $y$, we can decrease
$\ep$ and increase $N,T_0,\dots, T_N$ arbitrarily so that $g(t)$
remains a $d$-pseudotrajectory.

Similarly, (\ref{ST1}) holds if $\tau\in [0,T]$ and $\tau+t>T$.

 Thus $g(t)$ is ${\cal L}d-$shadowed by a trajectory $x(t)$ so that in particular\newline
dist$(x(t),x_0) \le{\cal L}d$ if $|t|$ is sufficiently large so that $x(t)\to x_0$ as $|t|\to\infty$.
\medskip

 We must also be sure that $x(t)\neq x_0$.  If $y$ is not on the local stable manifold of $x_0$,
then there exists $\varepsilon_1>0$ independent of $y$ such that
${\rm dist}(\phi(t_0,y),x_0)\ge\varepsilon_1$ for some $t_0>0$. We
can choose $T_0>t_0$. Now we know that ${\rm
dist}(x(t),\phi(t_0,y))\le {\cal L}d$. So provided ${\cal
L}d<\varepsilon_1$, we have $x(t_0)\neq x_0$.

If $y$ is on the local stable manifold of $x_0$, then provided ${\rm dist}(y,x_0)$ is sufficiently
small, it is not on the local unstable manifold of $x_0$.
Then, applying the same argument to the flow with time reversed noting that the chain recurrent set is also the chain recurrent set  for the reversed flow and also that the reversed flow will have the Lipschitz shadowing property also, we show that $x(t)\neq x_0$.
 \medskip

Now we show the existence of this homoclinic orbit $x(t)$ leads to a
contradiction. Set $p_k=x(k)$. Since $A_kX(p_k)=X(p_{k+1})$, it is
easily verified that if
$$ \beta_{k+1}=\beta_k+s_k,\quad k\in \ZZ$$
then $v_k=\beta_kX(p_k)$ is a solution of
\begin{equation}\label{27} v_{k+1}=A_kv_k +X(p_{k+1})s_k,\quad k\in \ZZ. \end{equation}
Also if $s_k$ is bounded then $\beta_kX(p_k)$ is also bounded, since
$X(p_k)\to 0$ exponentially as $|k|\to\infty$ and $|\beta_k|/|k|$ is
bounded.

By Lemma \ref{lemPalm2}, for all bounded $b_k\in{\cal M}_k$ there
exists a bounded scalar sequence $s_k$ such that
$$ v_{k+1}=A_kv_k+X(p_{k+1})s_k+b_{k+1}$$
has a bounded solution. But we know (\ref{27}) has a bounded solution. It follows that
$$ v_{k+1}=A_kv_k+b_{k+1}$$
has a bounded solution for arbitrary $b_k\in{\cal M}_k$. Then it
follows from Lemma \ref{lemPalm4} that
$$ v_{k+1}=A_kv_k$$
is hyperbolic on both $\ZZ_+$ and $\ZZ_-$ and that the spaces
$B^+(A)$ and $B^-(A)$ are transverse. This is a contradiction since
dim\,$B^+(A)+$ dim\,$B^-(A)=n$ (because $B^+(A)$ has the same
dimension as the stable manifold of $x_0$ and $B^-(A)$ has the same
dimension as the unstable manifold of $x_0$) but they contain
$X(p_0)\neq 0$ in their intersection.

So we conclude that the rest points are isolated in the chain recurrent set.
\end{proof}
\subsection{Hyperbolicity of the chain recurrent
set}\label{secPalm5}

We have shown that the rest points of $X$ are hyperbolic and form a finite,
isolated subset of the chain recurrent set ${\cal C}{\cal R}$. Let $\Sigma$ be the chain recurrent
set minus the rest points. We want to show this set is hyperbolic. To this
end we use the following lemma. Let us first introduce some notation.

    Let $x(t)$ be a trajectory of $X$ in $\Sigma$. Put $p_k = x(k)$ and denote by $P_k$
the orthogonal projection in ${\cal M}_k$ with kernel spanned by $X(p_k)$ and by $V_k$
the orthogonal complement to $X(p_k)$ in ${\cal M}_k$. Introduce the operators $B_k=
P_{k+1}A_k:V_k\mapsto V_{k+1}$.
\medskip

\begin{lem}\label{lemPalm6} For every bounded sequence $b_k\in V_k$
(denote $b=\|b\|_{\infty}$) there exists a solution $v_k\in V_k$ of
the system
\begin{equation}\label{14}  v_{k+1}=B_kv_k+b_{k+1},\quad k\in \ZZ, \end{equation}
such that for all $k$,
$$  |v_k| \le 2L(2b + 1).$$
\end{lem}

\begin{proof}   By Lemma \ref{lemPalm2}, there exists a bounded sequence $s_k$ such that
the system
\begin{equation}\label{15}  w_{k+1}=A_kw_k+X(p_{k+1})s_k+b_{k+1},\quad k\in \ZZ \end{equation}
has a solution $w_k$ with $|w_k|\le 2L(2b + 1)$.

Note that $A_kX(p_k)=X(p_{k+1})$. Since $(\Id-P_k)v\in  \{X(p_k)\}$
for $v\in {\cal M}_k$, we see that $P_{k+1}A_k(\Id-P_k)=0$, which
gives us the equality
\begin{equation}\label{16}    P_{k+1}A_k=P_{k+1}A_kP_k.  \end{equation}
Multiplying (\ref{15}) by $P_{k+1}$, taking into account the equalities $P_{k+1}X(p_{k+1})=0$
and $P_{k+1}b_{k+1}=b_{k+1}$, and applying (\ref{16}), we see that $v_k=P_kw_k$ is the required
solution.

Thus, the lemma is proved.
\end{proof}

Now we prove $\Sigma$ is hyperbolic. Let $x(t)$ be a trajectory in
$\Sigma$ with the same notation as given before Lemma
\ref{lemPalm6}. Then by Lemmas \ref{lemPalm6} and \ref{lemPalm4},
\begin{equation}\label{17} v_{k+1}=B_kv_k,\quad v_k\in V_k \end{equation}
is hyperbolic on both $\ZZ_+$ and $\ZZ_-$ and $B^+(B)$ and $B^-(B)$
are transverse. It follows that the adjoint system
$$ v_{k+1}=(B_k)^{*-1}v_k,\quad v_k\in V_k $$
is hyperbolic on both $\ZZ_+$ and $\ZZ_-$ and has no nontrivial
bounded solution.

Now we consider the discrete linear skew product flow on the normal
bundle $V$ over $\Sigma$ generated by the map defined for
$p\in\Sigma$, $v\in V_p$ (where $V_p$ is the orthogonal complement
to $X(p)$ in $T_pM$) by
\begin{equation}\label{Text11.1}
(p,v) \mapsto (\phi(1,p),B_pv),
\end{equation}
where $B_p=P_{\phi(1,p)}D\phi(1,p)$, $P_p$ being the orthogonal
projection of $T_pM$ onto $V_p$. Its adjoint flow is generated by
the map defined by
$$ (p,v) \mapsto (\phi(1,p),(B^*_p)^{-1}v).$$
Now we want to apply the Corollary on page 492 in Sacker and Sell \cite{SS1}.
What we have shown above is that the adjoint flow has the
no nontrivial bounded solution property. It follows from the Sacker and Sell corollary that
the adjoint flow is hyperbolic and hence the original
skew product flow
$$ (p,v) \mapsto (\phi(1,p),B_pv)$$
is also. However then it follows from Theorem 3 in Sacker and Sell \cite{SS2} that $\Sigma$ is hyperbolic.

\subsection{Strong Transversality}\label{secPalm6}

To verify strong transversality, let $x(t)$ be a trajectory that
belongs to the intersection of the stable and unstable manifolds of
two trajectories, $x_+(t)$ and $x_-(t)$, respectively, lying in the
chain recurrent set. Denote $p_0=x(0)$ and $p_k=x(k),\;k\in\ZZ$; let
$W^s(p_0)$ and $W^u(p_0)$ denote the stable manifold of $x_+(t)$ and
the unstable manifold of $x_-(t)$, respectively. Denote by $E^s$ and
$E^u$ the tangent spaces of $W^s(p_0)$ and $W^u(p_0)$ at $p_0$.

By Lemma \ref{lemPalm6} (using the same notation as in the previous
section), for all bounded $b_k\in V_k$, there exists a bounded
solution $v_k\in V_k$ of (\ref{14}). By Lemma \ref{lemPalm4} again,
this implies that
\begin{equation}\label{STC1}
{\cal E}^s+{\cal E}^u=V_0,
\end{equation}
where
$$
\begin{array}{rl}
 {\cal E}^s&=\{w_0: w_{k+1}=B_kw_k,\; \sup_{k\ge 0}|w_k|<\infty\},\\ \\
{\cal E}^u&=\{w_0: w_{k+1}=B_kw_k,\; \sup_{k\le
0}|w_k|<\infty\}\end{array}
$$
Moreover (\ref{17}) is hyperbolic on both $\ZZ_+$ and $\ZZ_-$.

We are going to use the following folklore result, which for
completeness we prove after showing it implies the strong
transversality:
\begin{equation}\label{STC2}
{\cal E}^s\subset E^s,\quad {\cal E}^u\subset E^u.
\end{equation}
Combining equality\ (\ref{STC1}) with the inclusions  (\ref{STC2})
and the trivial relations
$$
E^s = V_0 \cap E^s + \{X(p_0)\},\quad E^u = V_0 \cap E^u +
\{X(p_0)\},
$$
we conclude that
$$
E^s + E^u = T_{p_0}M,
$$
and so the strong transversality holds.

Let us now prove the first relation in (\ref{STC2}); the second one
can be proved in a similar way.

{\it Case 1}:  The limit trajectory in ${\cal C}{\cal R}$ is a
rest point. In this case, the stable manifold of the rest point
coincides with its stable
manifold as a fixed point of the time-one map $f(x)=\phi(1,x)$. By
the theory for diffeomorphisms, if $p_k$ is a trajectory on the
stable manifold, the tangent space to the stable manifold at $p_0$
is the subspace $E^s$ of initial values of bounded solutions of
\begin{equation}
\label{STC3}
v_{k+1} = A_k v_k, \quad k \geq 0.
\end{equation}

Let us prove that ${\cal E}^s \subset E^s$. Fix an arbitrary sequence
$w_k$ satisfying
$w_{k+1} = B_k w_k$ with $w_0 \in {\cal E}^s$.
Consider the sequence
$$
v_k = \lambda_kX(p_k)/|X(p_k)|+w_k
$$
with
$\lambda_k$ satisfying
\begin{equation}\label{21}
 \lambda_{k+1}=\displaystyle\frac{|X(p_{k+1})|}{|X(p_k)|}\lambda_k-
\frac{X(p_{k+1})^*}{|X(p_{k+1})|}A_kw_k
\end{equation}
and $\lambda_0 = 0$. It is easy to see that $v_k$ satisfy
(\ref{STC3}).

Since $x(t)$ is on the stable manifold of a hyperbolic rest point,
there are positive constants $K$ and $\alpha$ such that
$$
|\dot x(t)|\le Ke^{-\alpha(t-s)}|\dot x(s)|
$$
for $0\le s\le t$. From this it follows that
$$
|X(p_k)|\le Ke^{-\alpha(k-m)}|X(p_m)|
$$
for $0\le m\le k$ so that the scalar difference equation
$$
\lambda_{k+1}=\frac{|X(p_{k+1})|}{|X(p_k)|}\lambda_k
$$
is hyperbolic on $\ZZ_+$ and is, in fact, stable. Since the second
term on the right-hand side of equation (\ref{21}) is bounded as
$k\to\infty$, it follows that $\lambda_k$ are bounded for any choice
of $\lambda_0$. This fact implies that $v_k$ is a bounded solution
of (\ref{STC3}), and we conclude that $v_0 = w_0 \in E^s$, hence
${\cal E}^s \subset E^s$.

The proof in Case 1 is complete.
\medskip

{\it Case 2}: Assume that the limit trajectory is in $\Sigma$, the
chain recurrent set minus the fixed points which we know to be
hyperbolic. We want to find the intersection of its stable manifold
near $p_0=x(0)$ with the cross-section at $p_0$ orthogonal to the
vector field (in local coordinates generated by the exponential
mapping). To do this, we discretize the problem and note that there
exists a number $\sigma>0$ such that a point $p\in M$ close to $p_0$
certainly belongs to $W^s(p_0)$ if and only if the distances of
consecutive points of intersections of the positive semitrajectory
of $p$ with the sets $\exp_{p_k}({\cal M}_k)$ to the points $p_k$ do
not exceed $\sigma$.
\medskip

For suitably small $\mu>0$, we find all sequences $t_k$ and $z_k\in
V_k$, the subspace of $T_{p_k}M$ orthogonal to $X(p_k)$, such that
for $k\ge 0$
$$
|t_k-1|\le \mu,\quad |z_k|\le\mu,\quad  y_{k+1}=\phi(t_k,y_k),
$$
where $y_k=\exp_{p_k}(z_k)$. Thus we have to solve the equation
$$
\exp_{p_{k+1}}(z_{k+1})=\phi(t_k,\exp_{p_k}(z_k)),\quad k\ge 0
$$
for $t_k$ and $z_k\in V_k$ such that $|t_k-1|\le \mu$ and
$|z_k|\le\mu$.
\medskip

We set it up as a problem in Banach spaces. By lemmas \ref{lemPalm4}
and \ref{lemPalm6} the difference equation
$$
z_{k+1}=B_kz_k,\quad
    z_k\in V_k
$$
(recall that $B_k=P_{k+1}A_k$ and $P_k$ is the orthogonal projection
on ${\cal M}_k$ with range $V_k$), has an exponential dichotomy on
$\ZZ_+$ with projection (say) $Q_k:V_k\mapsto V_k$. Denote by ${\cal
R}(Q_0)$ the range of $Q_0$ and note that ${\cal R}(Q_0)={\cal
E}^s$. Fix a positive number $\mu_0$ and denote by $V$ the space of
sequences
$$
\{z_k\in V_k, |z_k|\leq \mu_0,k\in\ZZ_+\}
$$
and by $l^{\infty}(\ZZ_+, \{{\cal M}_{k+1}\})$ the space of
sequences $\{\zeta_k \in {\cal M}_{k+1}, \; k \in \ZZ_+ \}$ with the
usual norm.

Then the $C^1$ function
$$
G:[1-\mu_0,1+\mu_0]^{\ZZ_+}\times V\times {\cal R}(Q_0)\mapsto
\ell^{\infty}(\ZZ_+,\{{\cal M}_{k+1}\})\times {\cal R}(Q_0)
$$
given by
$$ G(t,z,\eta)=(\{ z_{k+1}-\exp^{-1}_{p_{k+1}}(\phi(t_k,\exp_{p_k}(z_k))\}_{k\ge 0},
Q_0z_0-\eta)
$$
is defined if $\mu_0$ is small enough.

We want to solve the equation
$$
G(t,z,\eta)=0
$$
for $(t,z)$ as a function of $\eta$. It is clear that
$$
G(1,0,0)=0,
$$
where the first argument of $G$ is $\{1,1,\ldots\}$, the second
argument is $\{0,0,\ldots\}$ and the right-hand side is
$(\{0,0,\ldots\},0)$.

To apply the implicit function theorem, we must verify that
$$
T=\frac{\partial G}{\partial (t,z)}(1,0,0)
$$
is invertible. Note that if $(s,w)\in \ell^{\infty}(\ZZ_+,\RR)\times
V$, then
$$
T(s,w)=(\{ w_{k+1}-X(p_{k+1})s_k-A_kw_k\}_{k\ge 0}, Q_0w_0).
$$
To show that $T$ is invertible, we must show that
$$
T(s,w)=(g,\eta)
$$
has a unique solution $(s,w)$ for all $(g,\eta)\in
l^\infty(\ZZ_+,\{{\cal M}_{k+1}\})\times {\cal R}(Q_0)$. So we need
to solve the equations
$$
w_{k+1}=A_kw_k+X(p_{k+1})s_k+g_k,\quad k\ge 0
$$
subject to
$$
Q_0w_0=\eta.
$$
If we multiply the difference equation by $X(p_{k+1})^*$ and solve
for $s_k$, we obtain
$$
s_k=-\frac{X(p_{k+1})^*}{|X(p_{k+1})|^2}[A_kw_k+g_k],\quad k\ge 0
$$
and if we multiply it by $P_{k+1}$, we obtain
$$
w_{k+1}=P_{k+1}A_kw_k+P_{k+1}g_k=B_kw_k+P_{k+1}g_k,\quad k\ge 0.
$$
Now we know this last equation has a unique bounded solution $w_k\in
V_k$, $k \geq 0$, satisfying $Q_0w_0=\eta$. Then the invertibility
of $T$ follows.
\medskip

Thus we can apply the implicit function theorem to show that there
exists $\mu>0$ such that provided $|\eta|$ is sufficiently small,
the equation $G(t,z,\eta)=0$ has a unique solution
$(t(\eta),z(\eta))$ such that $\|t-1\|_{\infty}\le\mu$,
$\|z\|_{\infty}\le\mu$. Moreover, $t(0)=1$, $z(0)=0$ and the
functions $t(\eta)$ and $z(\eta)$ are $C^1$.

The points $\exp_{p_0}(z_0(\eta))$ with small $|\eta|$ form a
submanifold containing $p_0$ and contained in $W^s(p_0)$. Thus, the
range of the derivative $z_0'(0)$ is contained in $E^s$.

Take an arbitrary vector $\xi\in {\cal E}^s$ and consider
$\eta=\tau\xi,\xi\in\RR$. Differentiating the equalities
$$
z_{k+1}(\tau\xi)=\exp_{p_{k+1}}^{-1}(\phi(t_k(\tau\xi),
\exp_{p_{k}}(z_k(\tau\xi)))),\quad k\geq 0,
$$
and
$$
Q_0z_0(\tau\xi)=\tau\xi
$$
with respect to $\tau$ at $\tau=0$, we see that
$$
s_k=\frac{\partial t_k}{\partial \eta}|_{\eta=0}\xi,\quad
w_k=\frac{\partial z_k}{\partial \eta}|_{\eta=0}\xi\in V_k
$$
are bounded sequences satisfying
$$
w_{k+1}=A_kw_k+X(p_{k+1})s_k, \quad Q_0w_0=\xi.
$$
Multiplying by $P_{k+1}$, we conclude that
$$
w_{k+1}=B_kw_k, k\geq 0,\quad Q_0w_0=\xi.
$$
It follows that $w_0\in{\cal E}^s={\cal R}(Q_0)$. Then
$w_0=Q_0w_0=\xi$. We have shown that the range of $z_0'(0)$ is
exactly ${\cal E}^s$, and thus ${\cal E}^s\subset E^s$.

%The intersection of the local stable manifold at $p_0$ with the
%cross-section orthogonal to the vector field (in local coordinates
%generated by the exponential mapping) is the set of points
%$\exp_{p_0}(z_0(\eta))$. Then the intersection of the tangent space
%to the local stable manifold at $p_0$ with the cross-section
%orthogonal to the vector field $X(p_0)$ is the range of the
%derivative $z'_0(0)$. Now
%$$
%z_{k+1}(\eta)=\exp^{-1}_{p_{k+1}}(\phi(t_k(\eta),\exp_{p_k}(z_k(\eta))),\quad k\ge 0,
%$$
%where $Q_0z_0(\eta)=\eta$. Differentiating with respect to $\eta$ at
%$\eta=0$, we find that $s_k=t'_k(0)$ and $w_k=z'_k(0)$ are bounded
%sequences satisfying
%$$
%w_{k+1}=X(p_{k+1})s_k+A_kw_k,\quad k\ge 0,
%$$
%where $Q_0w_0=Q_0$. Multiplying by $P_{k+1}$, we find that
%$$
%w_{k+1}=B_kw_k,\quad k\ge 0,
%$$
%where $Q_0Q_0=Q_0$. Thus the range of $w_0=z'_0(0)$ is the set of
%initial values of bounded solutions of
%$$
%w_{k+1}=P_{k+1}A_kw_k,\quad k\ge 0.
%$$
%Hence, the intersection of the tangent space to the local stable
%manifold at $x(0)$ with the cross-section orthogonal to the vector
%field $X(p_0)$ (in local coordinates generated by the exponential
%mapping) coincides with ${\cal E}^s$.

\section{Lipschitz periodic shadowing} \label{secLipPerSh}

It is known that a vector field $X$ is $\Omega$-stable if and only
$X$ satisfies Axiom A$'$ and the no-cycle condition (see
\cite{PughShub} and \cite{Hayashi}). Thus, to prove Theorem
\ref{thmLipPerSh}, we prove the following two lemmas.

\begin{lem}\label{lemLipPer1}
If a vector field $X$ has the Lipschitz periodic shadowing property,
then $X$ satisfies Axiom A$'$ and the no-cycle condition.
\end{lem}

\begin{lem}\label{lemLipPer2}
If $X$ satisfies Axiom A$'$ and the no-cycle condition, then $X$ has
the Lipschitz periodic shadowing property.
\end{lem}

Lemma \ref{lemLipPer1} is proved in Secs.
\ref{secLipPer1}-\ref{secLipPer5}; Lemma \ref{lemLipPer2} is proved
in Sec. \ref{secLipPer6}.

The proof of Lemma \ref{lemLipPer1} is divided into several steps.

We assume that $X$ has the Lipschitz periodic shadowing property and
establish the following statements.

\begin{enumerate}
\item Closed trajectories are uniformly hyperbolic.
\item Rest points are hyperbolic.
\item The chain-recurrent set coincides with the closure of the set of rest
points and closed trajectories; rest points are separated from the
remaining part of the chain-recurrent set.
\item The hyperbolic structure on the set of closed trajectories can be extended
to the chain-recurrent set.
\item The no-cycle condition holds.
\end{enumerate}

\subsection{Uniform hyperbolicity of closed trajectories}
\label{secLipPer1}

Without loss of generality we can assume that ${\cal L} > 1$.

Let $x(t)$ be a nontrivial closed trajectory of period $\omega$.
Choose $n_1,n \in \NN$ such that $\tau = n_1\omega/n \in [1/2, 1]$.
Let $x_k = x(k\tau)$, $f(x) = \phi(\tau, x)$, and $A_k = \D f(x_k)$.
Note that $A_{k+n} = A_k$. Below we prove a statement similar to
Lemma \ref{lemPalm2}.
\begin{lem}
\label{LP2} If $X \in \LipPerSh$, then for any $b>0$ there exists a
constant $K$ (the same for all closed trajectories $x(t)$ of $X$)
such that for any sequence $b_k \in T_{x_k} M$ with $|b_k| < b$
there exist sequences $s_k \in \RR$ and $v_k \in T_{x_k} M $ with
the following properties:
\begin{equation}
\label{2.1} v_{k+1} = A_k v_k + X(x_{k+1})s_{k+1} + b_{k+1}
\end{equation}
and
\begin{equation}
\label{2.2} |s_k|, |v_k| \leq K.
\end{equation}
\end{lem}

Before we go to the proof of Lemma \ref{LP2}, we need to generalize
the notion of discrete Lipschitz shadowing property. Let $d,\tau >
0$; we say that a sequence $y_k$ is a $\tau$-discrete
$d$-pseudotrajectory if $\dist(y_{k+1}, \phi(\tau, y_k)) < d$.

Let $\ep>0$; we say that a sequence $x_k$ $\ep$-shadows $y_k$ if
there exists a sequence $t_k > 0$ such that
$$
\dist(x_k, y_k)< \ep, \quad |t_k - \tau| < \ep, \quad x_{k+1}
=\phi(t_k, x_k).
$$

The following lemma can be proved similarly to Lemma \ref{lemPalm1}.

\begin{lem} If $X \in \LipPerSh$, then there exist constants $d_0, L > 0$ such
that for any $\tau \in [1/2, 1]$ and $d>0$ and any periodic
$\tau$-discrete $d$-pseudotrajectory $y_k$ with $d\leq d_0$ there
exists a sequence $x_k$ (not necessarily periodic) that $Ld$-shadows
$y_k$.
\end{lem}

In the proof of Lemma \ref{LP2}, we use the following technical
statement which is well-known in control theory
\cite{Sontag},\cite{Heemels}.

\begin{lem}
\label{LP2.5} Let $B:\RR^m \to \RR^m$ be a linear operator such that
the absolute values of its eigenvalues equal 1. Then for any
$\Delta_0 \in \RR^m$ and $\delta > 0$ there exists a number $R \in
\NN$ and a sequence $\delta_k \in \RR^m$, \;$k \in [1, R]$, such
that $|\delta_k| < \delta $ and the sequence $\Delta_k \in \RR^m$
defined by
\begin{equation}
\label{AC1.0.4} \Delta_{k+1} = B\Delta_k + \delta_{k+1}, \quad k \in
[0, R-1],
\end{equation}
satisfies $\Delta_R = 0$.
\end{lem}

\begin{proof}[Proof of Lemma \ref{LP2}]
Fix an arbitrary sequence $b_k$ with $|b_k| < b$ and a number $l \in
\NN$.

First we will find a number $l_1 > l$ and sequences $c_k$ and
$\Delta_k$ defined for $k \in [-ln, l_1n]$ such that $|c_k| < b$ and
$$
c_k = b_k, \quad k \in [-ln, ln],$$
\begin{equation}
\label{3.0.5} \Delta_{k+1} = A_k \Delta_k + c_{k+1}, \quad k \in
[-ln, l_1 n - 1],
\end{equation}
$$
\Delta_{-ln} = \Delta_{l_1n}.
$$
Consider the operator $A: T_{x_0} \to T_{x_0}$ defined by $A =
A_{n-1}\cdots A_0$.

The tangent space $T_{x_0}$ can be represented in the form
\begin{equation}
\label{3.1} T_{x_0} = E^s_0 \oplus E^c_0 \oplus E^u_0
\end{equation}
so that the subspace $E^s_0$ corresponds to the eigenvalues
$\lambda_j$ of $A$ such that $|\lambda_j|<1$, the subspace $E^c_0$
corresponds to the eigenvalues $\lambda_j$ such that
$|\lambda_j|=1$, and the subspace $E^u_0$ corresponds to the
eigenvalues $\lambda_j$ such that $|\lambda_j|>1$.

For any index $k$ consider the decomposition $T_{x_k} = E^s_k \oplus
E^c_k \oplus E^u_k$ as the image of decomposition \sref{3.1} under
the mapping $A_{k-1}\cdots A_0$.

In the coordinates corresponding to these decompositions, the
matrices $A_k$ can be represented in the following form:
$$
A_k = \diag(A^s_k, A^c_k, A^u_k).
$$
Set $A_\sigma=A^\sigma_{n-1}\cdots A^\sigma_{0}$ for $\sigma=s,c,u$.
Consider the corresponding coordinate representations $b_k = (b_k^s,
b_k^c, b_k^u)$, $c_k = (c_k^s, c_k^c, c_k^u)$, and $\Delta_k =
(\Delta_k^s, \Delta_k^c, \Delta_k^u)$ (and note that the values
$|b_k^s|, |b_k^c|, |b_k^u|$ are not necessarily less than $b$).

Equations \sref{3.0.5} are equivalent to the system
\begin{equation}
\label{4.2} \Delta^s_{k+1} = A_k^s \Delta^s_{k} + c_{k+1}^s,
\end{equation}
\begin{equation}
\notag\Delta^c_{k+1} = A_k^c \Delta^c_{k} + c_{k+1}^c,
\end{equation}
\begin{equation}
\label{4.3} \Delta^u_{k+1} = A_k^u \Delta^u_{k} + c_{k+1}^u.
\end{equation}

Set $c_k = b_k$ for $k \in [-ln, ln - 1]$.

Consider the sequence satisfying \sref{4.2} with  initial data
$\Delta^s_{-ln} = 0$ and denote $\Delta^s_{ln}$ by $a^s$; Consider
the sequence satisfying \sref{4.3} with  initial data $\Delta^u_{ln}
= 0$ and denote $\Delta^u_{-ln}$ by $a^u$.

There exist numbers $l_s, l_u > 1$ such that
$$
|A^{-l} a^u| < b, \; l \geq l_u, \quad |A^{l} a^s| < b, \; l \geq
l_s.
$$

Set $\Delta_{-ln} = (0, 0, a^u)$; then the definition of $a^s$ and
$a^u$ implies that $\Delta_{ln} = (a^s, C_1, 0)$ for some $C_1 \in
E^c_0$.

Set $c_k = 0$ for $k \in [ln+1, (l+l_s)n]$; then $\Delta_{(l+l_s)n}
= (A_s^{l_s}a^s, C_2, 0)$ for some $C_2 \in E^c_0$.

Set $c_k = 0$ for $k \in [(l+l_s)n+1, (l+l_s+1)n-1]$ and $c_k
=(-A_s^{l_s+1}a^s, 0, 0)$ for $k = (l+l_s+1)n$. Then
$\Delta_{(l+l_s+1)n} = (0, C_3, 0)$ for some $C_3 \in E^c_0$.

Applying Lemma~\ref{LP2.5} to $A^c:E^c_0\to E^c_0$, we find a number
$R$ and a sequence $\delta_k$ with $|\delta_k|\leq b$ such that if
$$
x_{i+1}=A^c x_i+\delta_{i+1},\quad x_0=\Delta^c_{(l+l_s+1)n},
$$
then $x_R=0$. Then if we set for $i=0,\dots,R-1$, $c_k=0$ for
$(l+l_s+i+1)n+1\leq k\leq (l+l_s+i+2)n-1$ and
$c_{(l+l_s+i+2)n}=(0,\delta_{i+1},0)$, we see that
$$
\Delta^c_{(l+l_s+i+2)n}=A^c\Delta^c_{(l+l_s+i+1)n}+\delta_{i+1},
\quad i=0,\dots,R-1,
$$
so that $\Delta^c_{(l+l_s+R+1)n}=0$; of course, the other two
components of $\Delta_{(l+l_s+R+1)n}$ remain zero.

Set $c_k = 0$ for $k \in [(l+l_s+R+1)n +1, (l+l_s+R+2)n -1]$ and
$c_k = (0, 0, A_u^{-l_u}a^u)$ for $k = (l+l_s+R+2)n$; then
$\Delta_{(l+l_s+R+2)n} = (0, 0, A_u^{-l_u}a^u)$. Finally, we set
$c_k =0$ for $k \in [(l+l_s+R+2)n +1, (l+l_s+R+2+l_u)n]$ and see
that $\Delta_{(l+l_s+2+R+l_u)n} = (0, 0, a^u) = \Delta_{-ln}$. Thus,
we have constructed the sequences mentioned in the beginning of the
proof.

Taking $d$ small enough, considering the periodic $\tau$-discrete
pseudotrajectory $y_k = \exp_{x_k}(d\Delta_k)$, and repeating the
reasoning similar to that in the proof of Lemma \ref{lemPalm2}, we
can prove that relations \sref{2.1} and \sref{2.2} hold with
$K=L(2b+1)$ for $k \in [-ln, ln-1]$.

After that, we repeat the reasoning used in the last two paragraphs
of the proof of Lemma \ref{lemPalm2} to complete the proof of Lemma
\ref{LP2}.
\end{proof}

As in Sec. \ref{secPalm5}, we define ${\cal M}_k, V_k, P_k$, and
$B_k = P_{k+1}A_k: V_k \to V_{k+1}$. Note that $B_k^{-1} =
P_kA_{k}^{-1}$. Since $M$ is compact, there exists a constant $N>0$
such that $\| \D \phi(\tau, x) \| < N$ for any $\tau \in [-1, 1]$
and $x \in M$. Hence, $\|A_k\|, \|A_k^{-1}\| < N$, and
\begin{equation}
\label{8.1} \|B_k\|, \|B_k^{-1}\| < N.
\end{equation}

The same reasoning as in the proof of Lemma~6 establishes the
following statement.

\begin{lem}
\label{LP8} There exists a constant $K > 0$ (the same for all closed
trajectories $x(t)$) such that for every sequence $b_k \in V_k$ with
$|b_k| \leq 1$ there exists a solution $v_k \in V_k$ of the system
$$
v_{k+1} = B_k v_k + b_{k+1}
$$
such that
$$
\|v_k\| \leq K.$$
\end{lem}

A remark on page 26 of \cite{Coppel}, Lemma \ref{LP8} and the
inequalities \sref{8.1} imply that there exist constants $C_1 > 0$
and $\lam_1 \in (0, 1)$ (the same for all closed trajectories) and a
representation $V_k = E^s(x_k) \oplus E^u(x_k)$ such that
$$
B_k E^s(x_k) = E^s(x_{k+1}), \quad B_k E^u(x_k) = E^u(x_{k+1}),
$$
$$
|B_{l+k}\cdots B_k v^s| \leq C_1 \lam_1^l|v_s|, \quad v^s\in
E^s(x_k), \; l>0, k \in \ZZ
$$
$$
|B^{-1}_{-l+k}\cdots B_k^{-1} v^u| \leq C_1\lam_1^l|v_u|, \quad v^u
\in E^u(x_k), \; l>0, k \in \ZZ.
$$

\begin{rem}
In fact in \cite{Coppel} exponential dichotomy with uniform
constants was proved only on $\ZZ^+$. However we can extend the
corresponding inequalities to the whole of $\ZZ$ by the periodicity
of $B_k$.
\end{rem}

Since $\tau \in [1/2, 1]$ and $\| \D \phi(\tau, x)\| \leq N$ the
above conditions imply that there exist constants $C_2
> 0$ and $\lam_2\in (0, 1)$ such that if $x(t)$ is a closed
trajectory, then
\begin{equation}
\label{8.5.1} |P_{\phi(t, x_0)} \D \phi(t, x(t_0)) v^s| \leq C_2
\lam_2^t |v^s|, \quad v^s \in E^s(x(t_0)), \; t>0, \; t_0 \in \RR,
\end{equation}
\begin{equation}
\label{8.5.2} |P_{\phi(-t, x_0)} \D \phi(-t, x(t_0)) v^u| \leq C_2
\lam_2^t |v^u|, \quad v^u \in E^u(x(t_0)), \; t>0, \; t_0 \in \RR,
\end{equation}
where $P_{y\in M}$ is the orthogonal projection of $T_yM$ with
kernel $X(y)$,  $E^{s, u}(x(t_0)) = P_{\phi(t_0, x)} \D \phi(t_0, x)
E^{s, u}(x_0)$.

\begin{rem}\label{LPR} In particular, the above
inequalities  imply that $x(t)$ is a hyperbolic closed trajectory.
\end{rem}

\subsection{Hyperbolicity of the rest points}

Let $x_0$ be a rest point. As in subsection \ref{secPalm3} (using
Lemma \ref{LP2}),
 we conclude that $\D \phi(1, x_0)$ is hyperbolic; hence, $x_0$ is a hyperbolic rest point.

\subsection{The rest points are separated from the remaining part of the
chain-recurrent set}

Denote by $\Per(X)$ the set of rest points and points belonging to
closed trajectories of a vector field $X$; let $\CR(X)$ be the set
of its chain-recurrent points. For a set $A \subset M$ denote by
$\Cl A$ the closure of $A$ and by $B(a,A)$ its $a$-neighborhood.

\begin{lem}
\label{LP7} If $X \in \LipPerSh$, then $\Cl \Per(X) = \CR(X)$.
\end{lem}

\begin{proof} If $y_0 \in \CR(X)$, then for any $d>0$
there exists a periodic $d$-pseudotrajectory $g(t)$ such that $g(0)
=y_0$.

Since $X\in\LipPerSh$, there exists a point $x_d \in \Per(X)$ such
that $\dist(x_d, y_0) < \LL d$. Hence, $B(\LL d, y_0) \cap \Per(X)
\ne \emptyset$ for arbitrary $d > 0$, which proves our lemma.
\end{proof}

\begin{lem}Let $X \in \LipPerSh$ and let $p$ be a rest point of $X$.
Then $p \notin \Cl(\CR(X)\setminus p)$.
\end{lem}

\begin{proof}
It has already been proved that all rest points of a vector field $X
\in \LipPerSh$ are hyperbolic; hence the set of rest points is
finite. Assume that $p \in \Cl(\CR(X)\setminus p)$. Then Lemma
\ref{LP7} implies that $p \in \Cl(\Per(X)\setminus p)$.

Denote by $W^s_{loc,a}(p)$ and $W^u_{loc,a}(p)$ the local stable and
unstable manifolds of size $a$.

Since the rest point $p$ is hyperbolic, there exists $\ep \in (0,
1/2)$ such that if $x \in M$ and $\phi(t,x)\subset B(4\ep,
p),\;t\geq 0$, then $x \in W^s_{loc,4\ep}(p)$; if $\phi(t,x)\subset
B(4\ep, p),\;t\leq 0$, then $x \in W^u_{loc,4\ep}(p)$; and if
$\phi(t,x)\subset B(4\ep, p),\;t\in\RR$, then $x=p$.

Let $d_1 = \min(d_0, \ep/\LL)$, where $d_0$ and $\LL$ are the
constants from the definition of $\LipPerSh$. Take a point $x_0 \in
\Per(X)$ (let the period of the trajectory of $x_0$ equal $\omega$)
and a number $T>0$ and define the mapping
$$
g_{x_0, T}(t) =
\begin{cases} p, & \quad t \in [-T, T],\\
\phi(t - T, x_0), & \quad t \in (T, T+\omega),\\
\end{cases}
$$
for $t \in [-T, T+\omega)$. Continue this mapping periodically to
the line $\RR$.

There exists $d_2 < d_1$ depending only on $d_1$ and $\nu$ (see
\sref{new5}) such that if $x_0 \in B(d_2, p)$, then $g_{x_0, T}(t)$
is a $d_1$-pseudotrajectory for any $T > 0$. We fix such a point
$x_0\in B(d_2,p)$ and consider below pseudotrajectories $g_{x_0,T}$
with this fixed $x_0$ and with increasing numbers $T$.

By our assumptions, the pseudotrajectory $g_{x_0, T}$ can be
$\ep$-shadowed by the trajectory of a point $z_T \in \Per(X)$ with
reparametrization $\al_T(t)$:
\begin{equation}
\label{11.1} \dist(g_{x_0, T}(t), \phi(\al_T(t), z_T)) < \ep.
\end{equation}

Our choice of $\ep$ implies that there exist times $t_1, t_2 > 0$
such that
$$
\dist(p, \phi(t_1, x_0)) \in [2\ep, 3\ep], \quad \phi(t, x_0) \in
B(4\ep, p), \quad t \in [0, t_1],
$$
$$
\dist(p, \phi(-t_2, x_0)) \in [2\ep, 3\ep], \quad \phi(t, x_0) \in
B(4\ep, p), \quad t \in [-t_2, 0].
$$
We emphasize that the numbers $t_1,t_2$ depend on our choice of the
point $x_0$ but not on our choice of $T$. Let
$$
r_T = \phi(\al_T(T+t_1), z_T), \quad q_T = \phi(\al_T(-T-t_2), z_T).
$$
Inequalities \sref{11.1} and the following two relations imply that
\begin{equation}
\label{11.2.1} \phi(t, q_T) \in B(5\ep, p), \quad t \in [0,
-\al_T(-T-t_2)],
\end{equation}
\begin{equation}
\label{11.2.2} \phi(t, r_T) \in B(5\ep, p), \quad t \in [ -
\al_T(T+t_1), 0].
\end{equation}

Since $\LL d_2\leq \ep <1/2$ and $t_1,t_2$ are fixed, inequality
\sref{eqdef2.1} implies that if $T$ is large enough, then
\begin{equation}
\label{11.3}  - \al_T(-T-t_2) \geq T/2, \quad \al_T(T+t_1)  \geq
T/2.
\end{equation}
Since \sref{11.2.1}-\sref{11.3} imply that $\dist(\phi(t, q_T), p)
\leq 4\ep$ for $0 \leq t \leq T/2$ and $\dist(\phi(t, r_T), p) \leq
4 \ep$ for $0 \geq t \geq -T/2$ it follows that
%Since the rest point $p$ is hyperbolic, it follows from relations
%\sref{11.2.1}-\sref{11.3} that
$$
\dist(q_T, W^s_{loc, 4\ep}(p)), \dist(r_T, W^u_{loc, 4\ep}(p)) \to
0, \quad T \to+\infty.
$$
Since $q_T, r_T \in B(4\ep, p) \setminus B(\ep, p)$, we can choose
sequences $q_n = q_{T_n} \to q$ and $r_n = r_{T_n} \to r$ such that
$q, r\ne p$, $q \in W^s_{loc,4\ep}(p)$, and $r \in
W^u_{loc,4\ep}(p)$.

Denote by $O(q_n)$ the (closed) trajectory of the point $q_n$.

From Remark \ref{LPR} we know that $O(q_n)$ is a hyperbolic closed
trajectory.

Passing to a subsequence, if necessary, we may assume that the
values $\dim W^s(O(q_n))$ are the same for all $n$. Since
$$
\dim W^s(O(q_n)) + \dim W^u(O(q_n))= \dim M +1
$$
and
$$
\dim W^s(p) + \dim W^u(p) = \dim M,
$$
we see that at least one of the following inequalities holds:
\begin{equation}\notag
 \dim W^s(O(q_n)) > \dim W^s(p)
\end{equation}
or
$$
\dim W^u(O(q_n)) > \dim W^u(p).
$$
Without loss of generality, we can assume that the first inequality
holds (in the other case we note that $O(q_n) = O(r_n)$ and consider
the vector field $-X$).

Denote $\sigma = \dim W^s(p)$. Consider the space $E^s_n = E^s(q_n)$
corresponding to inequalities \sref{8.5.1}, \sref{8.5.2}. Then the
following holds
$$
\dim E^s_n =\dim W^s(O(q_n)) - 1 \geq \sigma.
$$
Passing to a subsequence, if necessary, we may assume that $E^s_n
\to F^s \subset V_q$, where $V_q$ is the subspace in $T_q M$
orthogonal to $X(q)$ (here and below, we consider convergence of
linear spaces in the Grassman topology). Passing to the limit in
inequalities \sref{8.5.1}, we conclude that
\begin{equation}\notag
 |P_{\phi(t, q)} \D\phi(t, q) v^s| \leq C_2 \lam_2^t
|v^s|,
 \quad v^s\in F^s, \; t > 0.
\end{equation}

This inequality implies the inclusion $F^s \subset T W_q^s(p)$.
Hence,
$$
F^s \oplus \langle X(q) \rangle \subset T_q W^s(q),
$$
and $\dim W^s(q) \geq \sigma+1$. We get a contradiction which proves
Lemma \ref{LP8}.
\end{proof}

\subsection{Hyperbolicity of the chain-recurrent set}

Consider a point $y \in \CR(X)$ that is not a rest point. Lemma
\ref{LP7} implies that there exists a sequence $x_n \in \Per(X)$
such that $x_n \to y$.

Consider the decomposition $V_{x_n} = E^s(x_n) +E^u(x_n)$
corresponding to inequalities \sref{8.5.1}, \sref{8.5.2}. Denote
$E^{s, u}_n = E^{s, u}(x_n)$. Passing if necessary to a subsequence,
we may assume that the dimensions $\dim E^s_n$ and $\dim E^u_n$ are
the same for all $n$. Since $y$ is not a rest point, $V_{x_n} \to
V_y$.

Since inequalities \sref{8.5.1} and \sref{8.5.2} hold for all closed
trajectories with the same constants $C_2$ and $\lam_2$, standard
reasoning implies that the ``angles'' between $E^s_n$ and $E^u_n$
are uniformly separated from 0 (see, for instance,
\cite{PilSSBook}). So passing if necessary to a subsequence, we may
assume that $E^s_n \to E^s$ and $E^u_n \to E^u$.

Hence, $E^s \cap E^u = \{0\}$, $\dim (E^s + E^u) = \dim E^s + \dim
E^u = \dim V_y$, and $E^s + E^u = V_y$. Estimates \sref{8.5.1} and
\sref{8.5.2} for the points $x_n$ imply similar estimates for $y$.
Hence, the skew product flow \sref{Text11.1} is hyperbolic, and
Theorem 3 in Sacker and Sell \cite{SS2} implies that $\CR(X)$ is
hyperbolic.

\subsection{No-cycle condition}
\label{secLipPer5}

In the previous two subsections we have proved that the vector field
$X$ (and its flow $\phi$) satisfies Axiom A$'$. It is known that in
this case, the nonwandering set of $X$ can be represented as a
disjoint union of a finite number of compact invariant sets (called
basic sets):
\begin{equation}
\label{spe} \Omega(X)=\Omega_1\cup\dots\cup\Omega_m,
\end{equation}
where each of the sets $\Omega_i$ is either a hyperbolic rest point
of $X$ or a hyperbolic set on which $X$ does not vanish and which
contains a dense positive semi-trajectory.

The basic sets $\Omega_i$ have stable and unstable ``manifolds'':
$$
W^s(\Omega_i)=\{x\in M:\;\mbox{dist}(\phi(t, x),\Omega_i)\to 0,
\quad t\to\infty\}
$$
and
$$W^u(\Omega_i)=\{x\in M:\;\mbox{dist}(\phi(t, x),\Omega_i)\to 0,
\quad t\to-\infty\}.
$$
If $\Omega_i$ and $\Omega_j$ are basic sets, we write
$\Omega_i\to\Omega_j$ if the intersection
$$
W^u(\Omega_i)\cap W^s(\Omega_j)
$$
contains a wandering point.

We say that $X$ has a 1-cycle if there is a basic set $\Omega_i$
such that $\Omega_i\to\Omega_i$.

We say that $X$ has a $k$-cycle if there are $k>1$ basic sets
$$
\Omega_{i_1},\dots,\Omega_{i_k}
$$
such that
$$
\Omega_{i_1}\to\dots\to\Omega_{i_k}\to\Omega_{i_1}.
$$

\begin{lem}
If $X\in\LipPerSh$, then $X$ has no cycles.
\end{lem}

\begin{proof} To simplify the presentation, we prove that $X$ has no 1-cycles
(in the general case, the idea is essentially the same, but the
notation is heavy).

To get a contradiction, assume that
$$
p\in(W^u(\Omega_i)\cap W^s(\Omega_i))\setminus \Omega(X).
$$
Then there are sequences of times $j_m,k_m\to\infty$ as $m\to\infty$
such that
$$
\phi(-j_m, p), \phi(k_m, p)\to\Omega_i,\quad m\to\infty.
$$
Since the set $\Omega_i$ is compact, we may assume that
$$
\phi(-j_m, p)\to q\in\Omega_i\;\mbox{ and }\; \phi(k_m, p)\to
r\in\Omega_i.
$$
Since $\Omega_i$ contains a dense positive semi-trajectory, there
exist points $s_m\to r$ and times $l_m>0$ such that
$\phi(l_m,s_m)\to q$ as $m\to\infty$.

Clearly, if we continue the mapping
$$
g(t) = \begin{cases} \phi(t, p), & t \in [0, k_m],\\
\phi(t - k_m, s_m), & t \in [k_m, k_m + l_m],\\
\phi(t-j_m-k_m-l_m, p), & t \in [k_m+l_m, k_m + l_m +j_m],
\end{cases}
$$
periodically with period $k_m+l_m+j_m$, we get a periodic
$d_m$-pseudotrajectory of $X$ with $d_m\to 0$ as $m\to\infty$.

Since $X\in\LipPerSh$, there exist points $p_m \in \Per(X)$ (for $m$
large enough) such that $p_m\to p$ as $m\to\infty$, and we get the
desired contradiction with the assumption that $p\notin\Omega(X)$.
The lemma is proved.
\end{proof}

\subsection{$\Omega$-stability implies Lipschitz periodic
shadowing}\label{secLipPer6}

The proof of Lemma \ref{lemLipPer2} is similar to the corresponding
proof in \cite{OPT}, where the case of diffeomorphisms is
considered. In the present article we give the most important steps
and leave the details to the reader.

\begin{proof}[Proof of Lemma \ref{lemLipPer2}]

Let us formulate several auxiliary definitions and statements.

Let us say that a vector field $X$ has the Lipschitz shadowing
property on a set $U$ if there exist positive constants $\LL,d_0$
such that if $g(t)$ with $\{g(t):\;t\in\RR\}\subset U$ is a
$d$-pseudotrajectory (in our standard sense:
$$
\mbox{dist}(g(\tau+t),\phi(t,g(\tau)))<d,\quad \tau\in\RR,t\in[0,1])
$$
with $d\leq d_0$, then there exists a point $p\in U$ and a
reparametrization $\al$ satisfying inequality \sref{eqdef2.1} such
that
\begin{equation}
\label{os1} \mbox{dist}(g(t),\phi(\al(t),p))<\LL d,\quad t\in\RR.
\end{equation}

We say that a vector field $X$ is expansive on a set $U$ if there
exist positive numbers $a$ (expansivity constant) and $\delta$ such
that if two trajectories $\{\phi(t,p):\;t\in\RR\}$ and
$\{\phi(t,q):\;t\in\RR\}$ belong to $U$ and there exists a
continuous real-valued function $\alpha(t)$ such that
$$
\mbox{dist}(\phi(\alpha(t),q),\phi(t,p))\leq a,\quad t\in\RR,
$$
then $p=\phi(\tau,q)$ for some real $\tau \in (-\delta, \delta)$.

Let $X$ be an $\Omega$-stable vector field. Consider the
decomposition \sref{spe} of $\Omega(X)$. We will refer to the
following well-known statement \cite{Palmer1}.

\begin{thm}\label{thm4} If $\Omega_i$ is a basic set, then there exists a
neighborhood $U$ of $\Omega_i$ such that $X$ has the Lipschitz
shadowing property on $U$ and is expansive on $U$.
\end{thm}

We also need the following two lemmas. Analogs of these lemmas were
proved for diffeomorphisms in \cite{PilTarSak}; the proofs for flows
are the same.

\begin{lem}\label{lemOm3} For any
neighborhood $U$ of the nonwandering set $\Omega(X)$ there exist
positive numbers $B,d_1$ such that if $g(t)$ is a
$d$-pseudotrajectory of $\phi$ with $d\leq d_1$ and
$$
g(t)\notin U,\quad t\in[\tau,\tau+l],
$$
for some $l>0$ and $\tau\in\RR$, then $l\leq B$.
\end{lem}

%In fact, Lemma \ref{lemOm3} is an immediate corollary of the
%Birkhoff theorem: For any neighborhood $U$ of the nonwandering set
%$\Omega(X)$ there exists a positive number $B_0$ such that if
%$$
%g(t,x)\notin U,\quad t\in[\tau,\tau+l], \; x\in M,
%$$
%for some $l>0$ and $\tau\in\RR$, then $l\leq B_0$ (obtained by
%passage to the limit using the compactness of the phase space).

\begin{lem}\label{lemOm4}  Assume that the vector field $X$ is
$\Omega$-stable. Let $U_1,\dots,U_m$ be disjoint neighborhoods of
the basic sets $\Omega_1,\dots,\Omega_m$. There exist neighborhoods
$V_j\subset U_j$ of the sets $\Omega_j$ and a number $d_2>0$ such
that if $g(t)$ is a $d$-pseudotrajectory of $X$ with $d\leq d_2$,
$g(\tau)\in V_j$ and $g(\tau+t_0)\notin U_j$ for some
$j\in\{1,\dots,m\}$, some $\tau \in \RR$ and some $t_0>0$, then
$g(\tau+t)\notin V_j$ for $t\geq t_0$.
\end{lem}

%This lemma follows from a similar well-known statement for exact
%trajectories: Assume that the vector field $X$ is $\Omega$-stable.
%Let $U_1,\dots,U_m$ be disjoint neighborhoods of the basic sets
%$\Omega_1,\dots,\Omega_m$. There exist neighborhoods $V_j\subset
%U_j$ of the sets $\Omega_j$ such that if $\phi(\tau,x)\in V_j$ and
%$\phi(\tau+t_0,x)\notin U_j$ for some $j\in\{1,\dots,m\}$ and some
%$t_0>0$, then $\phi(\tau+t,x)\notin V_j$ for $t\geq t_0$, and from
%Lemma \ref{lemOm3} (the corresponding reduction for the case of
%diffeomorphisms can be found in \cite{PilTarSak}).

%\medskip

Now we pass to the proof itself.

Apply Theorem \ref{thm4} and Lemmas \ref{lemOm3}, \ref{lemOm4} and
find disjoint neighborhoods $W_1,\dots,W_m$ of the basic sets
$\Omega_1,\dots,\Omega_m$ such that
\begin{itemize}
\item[(i)]  $X$ has the Lipschitz shadowing property on each $W_j$ with
the same constants $\LL,d^*_0$;

\item[(ii)] $X$ is expansive on each $W_j$ with the same expansivity
constants $a$, $\delta$.
\end{itemize}

Find neighborhoods $V_j,U_j$ of $\Omega_j$ (and reduce $d^*_0$, if
necessary) so that the following properties are fulfilled:
\begin{itemize}
\item $V_j\subset U_j\subset W_j,\quad j=1,\dots,m$;

\item the statement of Lemma \ref{lemOm4} holds for $V_j$ and $U_j$ with some
$d_2>0$;

\item the $\LL d^*_0$-neighborhoods of $U_j$ belong to $W_j$.
\end{itemize}

Apply Lemma \ref{lemOm3} to find the corresponding constants $B,d_1$
for the neighborhood $V_1\cup\dots\cup V_m$ of $\Omega(X)$.

We claim that $X$ has the Lipschitz periodic shadowing property with
constants $\LL,d_0$, where
$$
d_0=\min\left(d^*_0,d_1,d_2,\frac{a}{2\LL}\right).
$$

Take a $\mu$-periodic $d$-pseudotrajectory $g(t)$ of $X$ with $d\leq
d_0$. Without loss of generality we can assume that $\mu > \delta$
(since $\mu$ is not necessarily the minimal period). Lemma
\ref{lemOm3} implies that there exists a neighborhood $V_j$ such
that the pseudotrajectory $g(t)$ intersects $V_j$; shifting time, we
may assume that $g(0)\in V_j$.

In this case, $\{g(t):\;t\in\RR\}\subset U_j$. Indeed, if
$g(t_0)\notin U_j$ for some $t_0$, then $g(t_0+k\mu)\notin U_j$ for
all $k$. It follows from Lemma \ref{lemOm4} that if $t_0+k\mu>0$,
then $g(t)\notin V_j$ for $t\geq t_0+k\mu$, and we get a
contradiction with the periodicity of $g(t)$ and the inclusion
$g(0)\in V_j$.

Thus, there exists a point $p$ such that inequalities (\ref{os1})
hold for some reparametrization $\al$ satisfying inequality
\sref{eqdef2.1}. Let us show that either $p$ is a rest point or the
trajectory of $p$ is closed. By the choice of $U_j$ and $W_j$,
$\phi(t,p)\in W_j$ for all $t\in\RR$. Let $q=\phi(\mu,p)$.

Inequalities (\ref{os1}) and the periodicity of $g(t)$ imply that
$$
\mbox{dist}(g(t),\phi(\al(t+\mu)-\mu,q))=
$$
$$
\mbox{dist}(g(t+\mu),\phi(\al(t+\mu),p))\leq \LL d,\quad t\in\RR.
$$
Thus,
$$
\mbox{dist}(\phi(\al(t),p),\phi(\al(t+\mu)-\mu,q))\leq 2\LL d\leq
a,\quad t\in\RR,
$$
which implies that
$$
\mbox{dist}(\phi(\theta,p),\phi(\be(\theta),q))\leq 2\LL d\leq
a,\quad \theta\in\RR,
$$
where $\be(\theta)=\al(\al^{-1}(\theta)+\mu)-\mu$.

Since $\phi(t,p)\in W_j$ for all $t\in\RR$, our expansivity
condition on $W_j$ implies that $q=\phi(\tau,p)$ for some $\tau \in
(-\delta, \delta)$.

This completes the proof.
\end{proof}


\begin{thebibliography}{99}

\bibitem{Palmer1}
K. J. Palmer, Shadowing in Dynamical Systems, Theory and
Applications. Kluwer, Dordrecht, 2000.

\bibitem{Pilyugin}
S. Yu. Pilyugin, Shadowing in Dynamical Systems, Lecture Notes in
Math., 1706, Springer, Berlin, 1999.

\bibitem{PilSSBook} S. Yu. Pilyugin, Introduction to Structurally Stable Systems of Differential
Equations, Birkhauser-Verlag, Basel, 1992.

\bibitem{Anosov}
D.V. Anosov, On a class of invariant sets of smooth dynamical
systems, Proc. 5th Int. Conf. on Nonlin. Oscill. 2, Kiev, (1970)
39-45.

\bibitem{Bowen}
R. Bowen,  Equilibrium States and the Ergodic Theory of Anosov
Diffeomorphisms, Lecture Notes Math., 470, Springer, Berlin, 1975.

\bibitem{Robinson}
C. Robinson,  Stability theorems and hyperbolicity in dynamical
systems, Rocky Mount. J. Math., 7 (1977) 425-437.

\bibitem{Morimoto}
A. Morimoto, The method of pseudo-orbit tracing and stability of
dynamical systems, Sem. Note, 39 (1979) Tokyo Univ.

\bibitem{Sawada}
K. Sawada,   Extended $f-$orbits are approximated by orbits, Nagoya
Math. J., 79 (1980) 33-45.

\bibitem{Pilyugin2}
S. Yu. Pilyugin, Variational shadowing, Disc. Cont. Dyn. Sys., 23
(2010) 733-737.

\bibitem{Pilyugin1}
S. Yu. Pilyugin, S.B. Tikhomirov, Lipschitz shadowing implies
structural stability, Nonlinearity, 23 (2010) 2509-2515.

\bibitem{Pilyugin4} S. Yu. Pilyugin, Shadowing in structurally stable
flows, J. Diff. Eqns., 140,  no. 2 (1997) 238-265.

\bibitem{Sak}
K. Sakai, Pseudo orbit tracing property and strong transversality of
diffeomorphisms of closed manifolds, Osaka J. Math., 31 (1994)
373-386.

\bibitem{PilRodSak}
S. Yu. Pilyugin, A.A. Rodionova, K. Sakai, Orbital and weak
shadowing properties, Disc. Cont. Dyn. Sys., 9 (2003) 287-308.

\bibitem{PilTikh}
Pilyugin, S.Yu., S.B. Tikhomirov, Vector fields with the oriented
shadowing property, J. Diff. Eqns., 248 (2010) 1345-1375.

\bibitem{LeeSak}
K. Lee, K. Sakai, Structural stability of vector fields with
shadowing, J. Diff. Eqns., 232 (2007) 303-313.

\bibitem{PilTikh2008}
Pilyugin, S.Yu., S.B. Tikhomirov, Sets of vector fields with various
shadowing properties of pseudotrajectories, Doklady Mathematics, 422
(2008) 30-31.

\bibitem{Tikh}
S.B. Tikhomirov,  Interiors of sets of vector fields with shadowing
properties that correspond to some classes of reparametrizations,
Vestnik St. Petersburg Univ. Math., 41,  no. 4 (2008) 360-366.

\bibitem{OPT}
A. V. Osipov, S. Yu. Pilyugin, S. B. Tikhomirov, Periodic shadowing
and $\Omega$-stability, Regul. Chaotic Dyn. 15, no. 2-3 (2010)
404--417.

\bibitem{MorSakSun}
K. Moriyasu, K. Sakai, W. Sun, $C^1$-stably expansive flows, J.
Diff. Eqns., 213, (2005) 352-367.

\bibitem{Robinson1}
C. Robinson, Structural stability of vector fields, Ann. Math., 99
(1974) 154-175.

\bibitem{Maizel}
A. D. Maizel', On stability of solutions of systems of differential
equations, Ural. Politehn. Inst. Trudy, 51 (1954) 20-50.

\bibitem{Coppel}
W. A. Coppel, Dichotomies in Stability Theory, Lecture Notes in
Math., Vol. 629, Springer-Verlag, Berlin, 1978.

\bibitem{Pliss}
V. A. Pliss,  Bounded solutions of nonhomogeneous linear systems of
differential equations, Probl. Asympt. Theory Nonlin. Oscill.,
Naukova Dumka, Kiev, (1977) 168-173.

\bibitem{Palmer2}
K. J. Palmer, Exponential dichotomies and transversal homoclinic
points, J. Diff. Eqns., 55 (1984) 225-256.

\bibitem{Palmer3}
K. J. Palmer, Exponential dichotomies and Fredholm operators, Proc.
Amer. Math. Soc., 104 (1988) 149-156.

\bibitem{Pilyugin3}
S. Yu. Pilyugin, Generalizations of the notion of hyperbolicity, J.
Difference Eqns. Appl., 12 (2006) 271-282.

\bibitem{SS1}
R. J. Sacker,  G.R.Sell, Existence of dichotomies and invariant
splittings for linear differential systems II, J. Diff. Eqns., 22
(1976) 478-496.

\bibitem{SS2}
R. J. Sacker, G. R. Sell, Existence of dichotomies and invariant
splittings for linear differential systems III, J. Diff. Eqns., 22
(1976) 497-522.

\bibitem{PughShub} C. Pugh, M. Shub, The $\Omega$-stability theorem for
flows, Invent. Math., 11 (1970) 150-158.

\bibitem{Hayashi} S. Hayashi, Connecting invariant manifolds and the
solution of the $C^1$-stability and $\Omega$-stability conjectures
for flows, Ann. Math., 145 (1997) 81-137.

\bibitem{Sontag}
Eduardo D. Sontag,   An algebraic approach to bounded
controllability of linear systems, Internat. J. Control, 39, no. 1
(1984) 181-188.

\bibitem{Heemels}
W.P.M.H. Heemels, M.K. Camlibel, Null Controllability of
Discrete-time Linear Systems with Input and State Constraints,
Proceedings of the 47th IEEE Conference on Decision and Control
Cancun, Mexico, Dec. 9-11, 2008.

\bibitem{PilTarSak} S. Yu. Pilyugin, K. Sakai, O. A. Tarakanov,  Transversality
properties and $C^1$-open sets of diffeomorphisms with weak
shadowing, Disc. Cont. Dyn. Syst., 9 (2003) 287-308.
























%\bibitem{Franke}
%J.E. Franke, J.F. Selgrade : Hyperbolicity and chain recurrence, J.
%Diff. Eqns.26 (1977) 27-36.








\end{thebibliography}
\end{document}